\documentclass[a4paper, 12pt]{article}
\usepackage{}

\usepackage{mathrsfs}
\usepackage{epsfig}
\usepackage{amsmath}
\usepackage{amssymb}
\usepackage{latexsym}
\usepackage{amsfonts}
\usepackage{amsthm}

\usepackage[numbers,sort&compress]{natbib}
\usepackage{graphicx}
\ifx\pdfoutput\undefined  \else
 \fi

\usepackage[centerlast]{caption2}

\usepackage{color}

\usepackage{txfonts}
\usepackage{amssymb}
\usepackage{mathrsfs}
\usepackage{amsfonts}

\newtheorem{theorem}{Theorem}[section]
\newtheorem{lemma}[theorem]{Lemma}

\newtheorem{prop}[theorem]{Proposition}

\newtheorem{definition}[theorem]{Definition}

\usepackage{latexsym,amssymb}
\usepackage{amsmath}

\newcommand{\F}{{\mathbb F}}

\begin{document}

\title{{Structure of the largest idempotent-product free sequences in semigroups}}
\author{
Guoqing Wang\\
\small{Department of Mathematics, Tianjin Polytechnic University, Tianjin, 300387, P. R. China}\\
\small{Email: gqwang1979@aliyun.com}
\\
}
\date{}
\maketitle

\begin{abstract}

Let $\mathcal{S}$ be a finite semigroup, and let $E(\mathcal{S})$ be the set of all idempotents of $\mathcal{S}$. Gillam, Hall and Williams proved in 1972 that every $\mathcal{S}$-valued sequence $T$ of length at least $|\mathcal{S}|-|E(\mathcal{S})|+1$ is not (strongly) idempotent-product free, in the sense that it contains a nonempty subsequence the product of whose terms, in their natural order in $T$, is an idempotent, which affirmed a question of Erd\H{o}s. They also showed that the value $|\mathcal{S}|-|E(\mathcal{S})|+1$ is best possible.

Here, motivated by Gillam, Hall and Williams' work, we determine the structure of the idempotent-product free sequences of length  $|\mathcal{S}\setminus E(\mathcal{S})|$ when the semigroup $\mathcal{S}$ (not necessarily finite) satisfies $|\mathcal{S}\setminus E(\mathcal{S})|$ is finite,
and we introduce a couple of structural constants for semigroups that reduce to the classical Davenport constant in the case of finite abelian groups.
\end{abstract}


\section {Introduction}

Let $\mathcal{S}$ be a nonempty semigroup, endowed with a binary associative operation $*$ on  $\mathcal{S}$, and denote by $E(\mathcal{S})$ the set of idempotents of $\mathcal{S}$, where $x\in \mathcal{S}$ is said to be an idempotent (in $\mathcal{S}$) if $x*x=x$.
Our interest in semigroups and idempotents comes from the following question of P. Erd\H{o}s to D.A. Burgess \cite{Burgess69}:

{\sl If $\mathcal{S}$ is a finite nonempty semigroup of order $n$, does any $\mathcal{S}$-valued sequence $T$ of length $n$ contain a nonempty subsequence the product of whose terms, in any order, is an idempotent?}

In 1969, Burgess \cite{Burgess69} gave an answer to this question in the case that $\mathcal{S}$ is commutative or contains only one idempotent. Shortly after, this question was completely affirmed by
D.W.H. Gillam, T.E. Hall and N.H. Williams, who actually proved the following stronger result:

\noindent \textbf{Theorem A.} (\cite{Gillam72}) \ {\sl Let $\mathcal{S}$ be a finite nonempty semigroup. Any $\mathcal{S}$-valued sequence of length  $|\mathcal{S}|-|E(\mathcal{S})|+1$ contains one or more terms whose product (in their natural order in this sequence) is an idempotent; In addition, the bound $|\mathcal{S}|-|E(\mathcal{S})|+1$ is optimal.}

That better bounds can be obtained, at least in principle, for specific classes of semigroups is somewhat obvious and, in any case, will be explained later, in Section 4.

Let $\mathcal{S}$ be a nonempty semigroup and $T$  a sequence of terms from $\mathcal{S}$. We call $T$ (weakly) {\bf idempotent-product free} if $T$ contains no nonempty subsequence the product whose terms, in any order, is an idempotent, and we call $T$ {\bf strongly idempotent-product free} if $T$ contains no nonempty subsequence the product whose terms, in their natural order in $T$, is an idempotent.

In fact, by using almost the same idea of arguments employed by Gillam, Hall and Williams \cite{Gillam72}, we can derive the following proposition for any semigroup $\mathcal{S}$ such that $|\mathcal{S}\setminus E(\mathcal{S})|$ is finite.
For the readers' convenience, we shall give the arguments in Section 3.

\begin{prop}\label{proposition inifinity}\label{proposition infinite semigroup} \ Let $\mathcal{S}$ be a nonempty semigroup such that $|\mathcal{S}\setminus E(\mathcal{S})|$ is finite. Then any $\mathcal{S}$-valued sequence of length $|\mathcal{S}\setminus E(\mathcal{S})|+1$ is not strongly idempotent-product free.
\end{prop}

So, a natural question arises:

{\sl If $\mathcal{S}$ is a nonempty semigroup such that $|\mathcal{S}\setminus E(\mathcal{S})|$ is finite, and
$T$ is a weakly (respectively, strongly) idempotent-product free $\mathcal{S}$-valued sequence of length $|\mathcal{S}\setminus E(\mathcal{S})|$, what can we say about $T$ and the structure of $\mathcal{S}$? }

In this manuscript, we completely answered this question in case that $T$ is a weakly idempotent-product free $\mathcal{S}$-valued sequence of length $|\mathcal{S}\setminus E(\mathcal{S})|$.
For the sake of exposition, we shall present the main theorem together with its proof in Section 3. Section 2 contains only some necessary preliminaries. In the final Section 4,
further researches are proposed.

\section{Some Preliminaries}

We begin by recalling some notations extensively used in zero-sum theory, though mostly in the setting of commutative groups, see (\cite{GH}, Chapter 5) for abelian groups and see \cite{Grynkiewicz13} for nonabelian groups.

Let $\mathcal{S}$ be a nonempty semigroup. Finite $\mathcal{S}$-valued sequences can be regarded as words in the free monoid $\mathcal{F}(\mathcal{S})$ with basis  $\mathcal{S}$, we denote them multiplicatively, so as to write $x_1x_2\cdots x_{\ell}$ in place of $(x_1,x_2,\ldots,x_{\ell})$, and call them simply sequences. We say the sequence $T=x_1x_2\cdots x_{\ell}\in \mathcal{F}(\mathcal{S})$ has length $|T|=\ell$. We say $T'=x_{i_1}x_{i_2}\cdots x_{i_t}$ is a subsequence of $T$ provided that $t\in [0,\ell]$ and $1\leq i_1<i_2<\ldots<i_t\leq \ell$.
Note that the operation (connecting two sequences) of $\mathcal{F}(\mathcal{S})$ is represented by $\cdot$, which is different from the operation of $\mathcal{S}$. Accordingly, we write $x^n$ for the $n$-fold product of an element $x\in \mathcal{S}$, and $T^{[n]}$ for the $n$-fold product of the sequence $T\in \mathcal{F}(\mathcal{S})$. By $TT'^{[-1]}$ we denote the remaining subsequence of $T$ obtained by deleting the terms of $T'$ from $T$.
 For any element $x\in \mathcal{S}$, by ${\rm v}_x(T)$ we denote the multiplicity of $x$ in the sequence $T$, i.e., the times which $x$ appears to be terms in the sequence $T$. We set ${\rm supp}(T)=\{x\in \mathcal{S}: {\rm v}_x(T)>0\}$.
Let $\sigma$ be any permutation of $\{1,2,\ldots,\ell\}$.
By $\pi_{\sigma}(T)$ we denote the product $x_{\sigma(1)}*x_{\sigma(2)}*\cdots*x_{\sigma(\ell)}$ of terms of $T$ in the order under the permutation $\sigma$. If $\sigma$ is the identity permutation,  we just write $\pi(T)$ simply for $\pi_{\sigma}(T)$. Let $$\prod(T)=\{\pi_{\sigma}(T'): T' \mbox{takes every nonempty subsequence of }T$$
$$\ \ \ \ \ \ \ \ \ \ \ \ \ \ \ \ \ \ \ \ \ \ \ \  \ \ \ \ \ \ \ \ \ \ \ \ \ \mbox{ and } \sigma \mbox{ takes every permutation of } [1,|T'|]\}.$$
We call $T$ a {\sl (weakly) idempotent-product free} $\mathcal{S}$-valued sequence by meaning that $$\prod(T)\cap E(\mathcal{S})=\emptyset,$$ and $T$ a {\sl strongly idempotent-product free} $\mathcal{S}$-valued sequence by meaning that $$\{\pi(T'): T' \mbox{takes every nonempty subsequence of }T\}\cap E(\mathcal{S})=\emptyset.$$  For any element $x$ of $\mathcal{S}$,
we define $$\lambda_T(x)=|\prod(T\cdot x)\setminus \prod(T)|.$$

The zero element of
$\mathcal{S}$, denoted $0_{\mathcal{S}}$ (if it exists), is the
unique element $z$ of $\mathcal{S}$ such that $z*x=x*z=z$ for every
$x\in \mathcal{S}$.

Let $X$ be a subset of $\mathcal{S}$. We say $X$ generates $\mathcal{S}$, or the elements of $X$ are generators of $\mathcal{S}$, provided that every element $s\in \mathcal{S}$ is the product of one or more elements in $X$, in which case we write $\mathcal{S}=\langle X\rangle$. In particular, we use $\langle x\rangle$ in place of $\langle X\rangle$ if $X=\{x\}$, and we say that $\mathcal{S}$ is a cyclic semigroup if it is generated by a single element.
For any element $x\in \mathcal{S}$ such that $\langle x\rangle$ is finite, the least integer $r>0$ such that $x^r=x^t$ for some positive integer $t\neq r$ is the {\bf index} of $x$, denoted $\mathcal{I}(x)$, then the least integer $k>0$ such that $x^{\mathcal{I}(x)+k}=x^{\mathcal{I}(x)}$ is the {\bf period} of $x$, denoted $\mathcal{P}(x)$.
Let $I$ be an ideal of the  semigroup $\mathcal{S}$, the relation defined by $$a \ \mathscr{I} \ b \Leftrightarrow a=b \mbox{ or }a,b\in I$$ is a congruence on $\mathcal{S}$, the Rees Congruence of the ideal $I$. The quotient semigroup $\mathcal{S}/I=\mathcal{S}/\mathscr{I}$ is the Rees quotient of $\mathcal{S}$ by $I$.
Let $Q$ be a semigroup with zero disjoint from $\mathcal{S}$. An {\bf ideal extension} of $\mathcal{S}$ by $Q$ is a semigroup $B$ such that $\mathcal{S}$ is an ideal of $B$ and the Rees quotient $B/\mathcal{S}=Q$. A {\bf partial homomorphism} of $Q^*=Q\setminus \{0_Q\}$ into a semigroup $\mathcal{D}$ is a mapping $f: Q^* \rightarrow \mathcal{D}$ such that $f(a*b)=f(a)*f(b)$ whenever $a*b\neq 0_Q$.

If $\mathcal{S}$ is a commutative semigroup, it is then possible to define a fundamental congruence, $\mathcal{N}_{\mathcal{S}}$, on $\mathcal{S}$ as follows: Let $a,b$ be any two elements of $\mathcal{S}$. We write $a \leqq_{\mathcal{N}_{\mathcal{S}}} b$ to mean that $a^m=b*c$ for some $c\in \mathcal{S}$ and some integer  $m>0$. If $a \leqq_{\mathcal{N}_{\mathcal{S}}} b$ and $b \leqq_{\mathcal{N}_{\mathcal{S}}} a$, we write $a \ \mathcal{N}_{\mathcal{S}} \ b$. We call the commutative semigroup $\mathcal{S}$ an archimedean semigroup provided that $a \ \mathcal{N}_{\mathcal{S}} \  b$ for any two elements $a,b$ of $\mathcal{S}$.
By (\cite{Grillet monograph}, Chapter III,
Theorem 1.2), the quotient semigroup $Y(\mathcal{S})=\mathcal{S}/ \mathcal{N}_{\mathcal{S}}$ is a lower semilattice, called the {\bf universal semilattice} of $\mathcal{S}$. Furthermore,
there exists a partition
$\mathcal{S}=\bigcup_{y\in Y(\mathcal{S})} \mathcal{S}_y$ into subsemigroups $\mathcal{S}_y$ (one for every $y\in Y(\mathcal{S})$) with respect to the universal semilattice $Y(\mathcal{S})$, in particular,
$\mathcal{S}_{y_1}*\mathcal{S}_{y_2}\subseteq \mathcal{S}_{y_1\wedge y_2}$ for all $y_1,y_2\in Y(\mathcal{S})$, and each
component
$\mathcal{S}_y$ is archimedean. The following lemma to characterize the structure of any finite commutative archimedean semigroup will be useful for the proof later.

\begin{lemma}
\label{lemma structure of Archimedean} (\cite{Grillet monograph}, Chapter I, Proposition 3.6, Proposition 3.7,
Proposition 3.8, and Chapter III,
Proposition 3.1) \ A finite commutative semigroup $\mathcal{S}$ is archimedean if and only if it is an ideal extension of a finite abelian group $G$ by a
finite commutative nilsemigroup $N$. Moreover, the partial homomorphism $\varphi_G^N: N\setminus \{0_N\}\rightarrow G$ to construct the ideal extension of the group $G$ by the nilsemigroup $N$ is given by $$\varphi_G^N: a\mapsto a*e_{G}$$ where $a$ denotes an arbitrary element $N\setminus \{0_N\}=\mathcal{S}\setminus G$ and $e_G$ denotes the identity element of the subgroup $G$.
\end{lemma}

We say that the semigroup $\mathcal{S}$ is a nilsemigroup if every element of $\mathcal{S}$ is nilpotent, i.e., $\mathcal{S}$ has a zero element $0_{\mathcal{S}}$ and for each element $x\in \mathcal{S}$ there exists an integer $n>0$ such that $x^n=0_{\mathcal{S}}$.

 The following lemmas  will be useful for our arguments.

\begin{lemma}(see \cite{Grillet semigroup}, Chapter IV, p. 127) \label{Lemma addition in nilsemigroup} \ Let $N$ be a finite commutative nilsemigroup,
and let $a,b$ be two elements of $N$. If $a*b\in \{a,b\}$, then $a=0_N$ or $b=0_N$.
 \end{lemma}

\begin{lemma}\label{Lemma cyclic semigroup} (\cite{Grillet monograph},  Chapter I, Lemma 5.7, Proposition 5.8, Corollary 5.9) \ Let $\mathcal{S}=\langle x\rangle$ be a finite cyclic semigroup. Then  $\mathcal{S}=\{x,x^2,\ldots,x^{\mathcal{I}(x)},x^{\mathcal{I}(x)+1},\ldots,x^{\mathcal{I}(x)+\mathcal{P}(x)-1}\}$
with
$$\begin{array}{llll} & x^i *x^j=\left \{\begin{array}{llll}
               x^{i+j}, & \mbox{ if } \  i+j \leq  \mathcal{I}(x)+\mathcal{P}(x)-1;\\
                x^{k}, &  \mbox{ if }  \ i+j \geq \mathcal{I}(x)+\mathcal{P}(x), \ \mbox{ where} \\
                &  \ \ \ \ \ \ \ \ \ \  \mathcal{I}(x)\leq k\leq \mathcal{I}(x)+\mathcal{P}(x)-1 \ \mbox{ and } \ k\equiv i+j\pmod{\mathcal{P}(x)}. \\
              \end{array}
           \right. \\
\end{array}$$
Moreover,

\noindent (i) \  there exists a unique idempotent, $x^{\ell}$, in the cyclic semigroup $\langle x\rangle$, where $$\ell\in [\mathcal{I}(x),\mathcal{I}(x)+\mathcal{P}(x)-1] \  \mbox{ and }\  \ell\equiv 0\pmod {\mathcal{P}(x)};$$

\noindent (ii) \  $\{x^{\mathcal{I}(x)},x^{\mathcal{I}(x)+1},\ldots,x^{\mathcal{I}(x)+\mathcal{P}(x)-1}\}$ is a cyclic subgroup of $\mathcal{S}$ isomorphic to the additive group $\mathbb{Z}_{\mathcal{P}(x)}$ of integers modulo $\mathcal{P}(x)$.
\end{lemma}

\section{The structure of the extremal sequence}

In this section, we shall determine the structure of idempotent-product free $\mathcal{S}$-value sequences of length $|\mathcal{S}\setminus E(\mathcal{S})|$.
The following lemma will be useful.

\begin{lemma}\label{lemma lambda} Let $\mathcal{S}$ be a nonempty semigroup. Let $T$ be an $\mathcal{S}$-valued sequence with $\prod(T)\cap E(\mathcal{S})=\emptyset$, and let $x$ be a term of $T$. Then
$\lambda_{T x^{[-1]}}(x)\geq 1.$
\end{lemma}
\begin{proof} Since $|\prod(T)|$ is finite, combined with Lemma \ref{Lemma cyclic semigroup} (i), we derive that $\langle x\rangle\nsubseteq \prod(T)$ no matter whether $\langle x\rangle$ is finite or infinite, and thus, $\langle x\rangle\nsubseteq \prod(Tx^{[-1]})$. Let $k$ be the least positive integer such that $x^k\notin \prod(Tx^{[-1]})$. If $k=1$, i.e., $x\notin \prod(Tx^{[-1]})$, then $x\in \prod(T)\setminus \prod(Tx^{[-1]})$ which implies $\lambda_{Tx^{[-1]}}(x)\geq 1$, done. Hence, we assume $k>1$. Then $x^{k-1}\in \prod(T x^{[-1]})$, and thus, $x^k=x^{k-1}*x\in \prod(T x^{[-1]})*x\subseteq \prod(T)$, which implies $\lambda_{T x^{[-1]}}(x)\geq 1$. This completes the proof.
\end{proof}

\medskip

\noindent {\sl Proof of Proposition \ref{proposition inifinity}} \ Let  $T=a_1a_2\cdots a_{\ell}\in \mathcal{F}(\mathcal{S})$ with length $\ell=|\mathcal{S}\setminus E(S)|+1$. where $a_i\notin E(\mathcal{S})$. Suppose to the contrary that $T$ is strongly idempotent-product free.
Let $$A_k=\{\pi(T_k): T_k \mbox{ is a nonempty subsequence of } a_1a_2\cdots a_k\}$$ where $k\in [1,\ell]$.  Clearly,
\begin{equation}\label{equation At weak increasing}
A_1\subseteq A_2\subseteq\cdots \subseteq A_{\ell}.
 \end{equation}

We shall prove that
\begin{equation}\label{equation At increasing}
|A_{t+1}|>|A_t| \mbox{ for each }t\in [1,\ell-1].
 \end{equation}
Since $|\mathcal{S}\setminus E(S)|$ is finite, we have that the cyclic subsemigroup $\langle a_{t+1}\rangle$ is finite and contains an idempotent. Let $m$ be the least positive integer such that $a_{t+1}^m\notin A_t$. If $m=1$ then $a_{t+1}\in A_{t+1}\setminus A_t$, and if $m>1$ then $a_{t+1}^m=a_{t+1}^{m-1}*a_{t+1}\in A_{t+1}\setminus A_t$, which implies \eqref{equation At increasing}.

By  \eqref{equation At weak increasing} and
\eqref{equation At increasing}, we conclude that $|A_{\ell}|\geq |A_1|+{\ell-1}=\ell=|\mathcal{S}\setminus E(\mathcal{S})|+1$, a contradiction with $T$ being strongly idempotent-product free. \qed

Now we are in a position to give the main theorem.

\begin{theorem}\label{theorem extremal-sequence} \ Let $\mathcal{S}$ be a nonempty semigroup such that $|\mathcal{S}\setminus E(\mathcal{S})|$ is finite, and let $T$ be an $\mathcal{S}$-valued sequence
of length $|\mathcal{S}\setminus E(\mathcal{S})|$. Then $\prod(T)\cap E(\mathcal{S})=\emptyset$ if, and only if,
$\mathcal{R}=\langle {\rm supp}(T)\rangle$ is a finite commutative semigroup with $\mathcal{S}\setminus \mathcal{R}\subseteq E(\mathcal{S})$ and the universal semilattice $Y(\mathcal{R})$ is a chain such that $x_1*x_2=x_1$ for any elements $x_1,x_2\in \mathcal{R}$ with $x_1 \ \lneqq_{\mathcal{N}_{\mathcal{R}}} \ x_2$, and moreover,

\noindent (i) \  each archimedean component of $\mathcal{R}$ is, either a finite cyclic semigroup $\langle x\rangle$ with $x\in {\rm supp}(T)$ and $\mathcal{I}(x)\equiv 1\pmod {\mathcal{P}(x)}$, or an ideal extension of a nontrivial finite cyclic group $\langle x_2\rangle$ by a nontrivial finite cyclic nilsemigroup $\langle x_1\rangle$ with $x_1,x_2\in {\rm supp}(T)$ and
the partial homomorphism $\varphi_{\langle x_2\rangle}^{\langle x_1\rangle}$ being trivial, i.e.,   $\varphi_{\langle x_2\rangle}^{\langle x_1\rangle}(x_1)=e_{\langle x_2\rangle}$ where $e_{\langle x_2\rangle}$ denotes the identity element of the subgroup $\langle x_2\rangle$;

\noindent (ii) \ ${\rm v}_x(T)=\mathcal{I}(x)+\mathcal{P}(x)-2$ for each element $x\in {\rm supp}(T)$.
\end{theorem}

\noindent {\sl Proof of Theorem \ref{theorem extremal-sequence}.} \ The sufficiency is easy to verify.  We need only to consider the necessity. Note first that the cyclic semigroup $\langle a\rangle$ is finite for every non-idempotent element $a\in \mathcal{S}$, since otherwise, $\langle a\rangle$ would be isomorphic to the additive semigroup $\mathbb{N}^{+}$,  which is a contradiction with $|\mathcal{S}\setminus E(\mathcal{S})|$ being finite.
Let $\ell=|T|=|\mathcal{S}\setminus E(S)|$ and $T=a_1a_2\cdots a_{\ell}\in \mathcal{F}(\mathcal{S})$ with $\prod(T)\cap E(\mathcal{S})=\emptyset$.
Let $\tau$ denote an arbitrary permutation of $\{1,2,\ldots,\ell\}$, and
let $$T_k^{\tau}=a_{\tau(1)}a_{\tau(2)}\cdots a_{\tau(k)}$$ for each $k\in [1,m]$.
Since $\prod(T_k^{\tau})\cap E(\mathcal{S})=\emptyset$ for all $k\in [1,\ell]$,
it follows from Lemma \ref{lemma lambda} that
$$\begin{array}{llll}
|T|&=&|\mathcal{S}\setminus E(\mathcal{S})|\\
&\geq & |\prod(T)|= |\prod(T_{\ell-1}^{\tau})|+\lambda_{T_{\ell-1}^{\tau}}(a_{\tau(\ell)})\\
&\geq& |\prod(T_{\ell-1}^{\tau})|+1=|\prod(T_{\ell-2}^{\tau})|+\lambda_{T_{\ell-2}^{\tau}}(a_{\tau(\ell-1)})+1\\ &\geq& |\prod(T_{\ell-2}^{\tau})|+2\\
&\vdots&\\
&\geq& |\prod(T_1^{\tau})|+\ell-1=\ell=|T|.\\
\end{array}$$
It follows that
\begin{equation}\label{equation sumTk}
|\prod(T_{k}^{\tau})|=k
\end{equation}
 for each $k\in [1,\ell]$,
 and that \begin{equation}\label{equation sum T=}
\prod(T)=\mathcal{S}\setminus E(\mathcal{S}).
\end{equation}

Then we have the following.

\noindent \textbf{Claim A.} \ If $a,b$ are two distinct elements of ${\rm supp}(T)$, then $a*b=b*a\in \{a,b\}$.

\noindent{\sl Proof of Claim A.} \ By \eqref{equation sumTk} and the arbitrariness of $\tau$, we have that $|\prod(a\cdot b)|=2$, which implies $a*b,\ b*a\in \{a,b\}.$ Suppose to the contrary without loss of generality that $a*b\neq b*a$ with $a*b=b$ and $b*a=a.$ It follows that $a*a=a*(b*a)=(a*b)*a=b*a=a$, and so $a$ is an idempotent, which is absurd. This proves Claim A. \qed

By Claim A, then $\mathcal{R}=\langle {\rm supp}(T)\rangle$ is {\bf commutative}. Moreover, we have the following.

\noindent \textbf{Claim B.} \ $$\mathcal{R}=\bigcup\limits_{a\ \in {\rm supp}(T)} \langle a\rangle.$$ In particular, for any $x\in \prod(T)$, there exists an element $a\in {\rm supp}(T)$ such that $x=a^{k}$ with $k\in [1,{\rm v}_a(T)]$.

\noindent {\sl Proof of Claim B.} \ Take an arbitrary element $x$ of $\mathcal{R}$.  There exists some distinct elements of ${\rm supp}(T)$,  say $x_1,x_2,\ldots,x_{m}$, such that $x=x_1^{n_1}*x_2^{n_2}*\cdots *x_{m}^{n_{m}},$ where $m>0$ and $n_1,n_2,\ldots,n_{m}>0$. By applying Claim A, we conclude that $x=x_t^{n_t}$ for some $t\in [1,m]$. In particular, if $x\in \prod(T)$, we can take all the integers $n_1,n_2,\ldots,n_{m}$ above such that $n_i\in [1,{\rm v}_{x_i}(T)]$ for every $i\in \{1,2,\ldots,m\}$.  This proves Claim B. \qed

By Claim B, we see that $\mathcal{R}$ is {\bf finite} and we have the following.

\noindent \textbf{Claim C.} \ For any $a\in {\rm supp}(T)$ and any integer $k\in [1,\mathcal{I}(a)+\mathcal{P}(a)-1]$ such that $a^k\in \prod(T)$, $${\rm v}_a(T)\geq k.$$

\noindent {\sl Proof of Claim C.} \  By Claim B, we have that $a^k=b^{t}$ for some  $b\in {\rm supp}(T)$ with $t\in [1,{\rm v}_b(T)].$ Suppose $b\neq a.$ It follows from  Claim A that $a^k*a^k=a^k*b^t=b^t=a^k$, and thus $a^k$ is an idempotent, a contradiction. Hence, $b=a$ and ${\rm v}_a(T)={\rm v}_b(T)\geq t\geq k$. This proves Claim C. \qed

Let $g$ and $h$ be two arbitrary elements of $\mathcal{R}$ which belong to two distinct archimedean components of $\mathcal{R}$. By Claim B, we have $g=a^k$ and $h=b^t$ where $a,b$ are distinct elements of ${\rm supp}(T)$ and $k,t>0$.
It follows from Claim A that $$g*h=a^k* b^t=a^k=g$$ or $$g*h=a^k * b^t=b^t=h$$ which implies
$$
g\lneqq_{\mathcal{N}_{\mathcal{R}}} h
$$
 or $$h\lneqq_{\mathcal{N}_{\mathcal{R}}} g.$$
Since $\mathcal{N}_{\mathcal{R}}$ is a congruence on $\mathcal{R}$, by the arbitrariness of $g$ and $h$, we conclude that the universal semilattice $Y(\mathcal{R})=\mathcal{R}\diagup \mathcal{N}_{\mathcal{R}}$ is a chain and $g*h=g$ for any elements $g,h\in \mathcal{R}$ with $g \ \lneqq_{\mathcal{N}_{\mathcal{R}}} \ h$.

Let $a$ be an arbitrary element of ${\rm supp}(T)$.
By \eqref{equation sum T=}, we have that all the elements except for the unique idempotent of $\langle a\rangle$  must belong to $\prod(T)$. Combined with Lemma \ref{Lemma cyclic semigroup} and Claim C, we conclude that
\begin{equation}\label{equation times of x in T}
{\rm v}_a(T)=\mathcal{I}(a)+\mathcal{P}(a)-2,
\end{equation}
 and that the unique idempotent in the cyclic semigroup $\langle a\rangle$ is $a^{\mathcal{I}(a)+\mathcal{P}(a)-1}$ which implies $\mathcal{I}(a)+\mathcal{P}(a)-1 \equiv 0\pmod {\mathcal{P}(a)},$ equivalently,
 \begin{equation}\label{equation I(x) mod P(x)}
 \mathcal{I}(a) \equiv 1\pmod {\mathcal{P}(a)}.
\end{equation}
By \eqref{equation times of x in T}, we have Conclusion (ii) proved.  Now it remains to show Conclusion (i).

Let $A_{y}$ ($y\in Y(\mathcal{R})$) be an arbitrary archimedean component of $\mathcal{R}$. Since $x \ \mathcal{N}_{\mathcal{R}} \ x^t$ for any element $x\in \mathcal{R}$ and any integer $t>0$,  we conclude by Claim B that
$A_y$ is a union of several cyclic subsemigroups generated by the elements of ${\rm supp}(T)$, i.e.,
\begin{equation}\label{equation every archimedean component}
A_{y}=\bigcup\limits_{i=1}^{k_y} \langle x_i\rangle,
 \end{equation}
 where $k_y\geq 1$ and $x_1,x_2,\ldots,x_{k_y}$ are distinct elements of ${\rm supp}(T)$.
 By Lemma \ref{lemma structure of Archimedean}, we may assume that $A_y$ is an ideal extension of a group $G_y$ by a nilsemigroup $N_y$ (note that $G_y$ or $N_y$ may be trivial which shall be reduced to the case that $A_y$ is a nilsemigroup or a group). Now we show that
 \begin{equation}\label{equation G_y leq 1}
 |G_y\cap {\rm supp}(T)|\leq 1
 \end{equation}
  and
 \begin{equation}\label{equation N_y leq 1}
 |(A_y\setminus G_y)\cap {\rm supp}(T)|\leq 1.
 \end{equation}
Suppose $a,b$ are two distinct elements of $A_y\cap {\rm supp}(T)$. Recalling Claim A, we see $$a*b\in \{a,b\}.$$ If $a,b\in G_y$, then $a$ or $b$ is the identity element of the group $G_y$ which is an idempotent, a contradiction. If $a,b\in A_y\setminus G_y=N_y\setminus \{0_{N_y}\}$,  by Lemma \ref{Lemma addition in nilsemigroup}, we derive a contradiction. This proves \eqref{equation G_y leq 1} and \eqref{equation N_y leq 1}.

 By \eqref{equation G_y leq 1} and \eqref{equation N_y leq 1}, we have that $$k_y\in\{1,2\}$$ in \eqref{equation every archimedean component}.

 Consider the case of $k_y=1$, i.e., $A_{y}=\langle x\rangle$ for some $x\in {\rm supp}(T)$. Combined with  \eqref{equation I(x) mod P(x)},  we have Conclusion (i) proved.

 Consider the case of $k_y=2$, i.e., $A_{y}=\langle x_1\rangle\cup \langle x_2\rangle$ where $x_1$ and $x_2$ are distinct elements of ${\rm supp}(T)$. By \eqref{equation G_y leq 1} and \eqref{equation N_y leq 1}, we may assume without loss of generality that $x_2\in G_y$ and $x_1\in A_y\setminus G_y=N_y\setminus \{0_{N_y}\}$. Combined with Claim A, we see $x_1*x_2=x_2$.
Then we conclude that
the partial homomorphism $\varphi_{\langle x_2\rangle}^{\langle x_1\rangle}$ is trivial, and $G_y=\langle x_2\rangle$ and $N_y=\langle x_1\rangle$, and so Conclusion (i) holds.

This completes the proof of Theorem \ref{theorem extremal-sequence}. \qed

It is not hard to see that Theorem \ref{theorem extremal-sequence} can be also stated as the following equivalent form.

{\sl Let $\mathcal{S}$ be a nonempty semigroup such that $|\mathcal{S}\setminus E(\mathcal{S})|$ is finite, and let $T$ be an $\mathcal{S}$-valued sequence
of length $|\mathcal{S}\setminus E(\mathcal{S})|$. Then $\prod(T)\cap E(\mathcal{S})=\emptyset$ if, and only if, $\mathcal{R}=\langle {\rm supp}(T)\rangle$ is a finite commutative semigroup such that $\mathcal{S}\setminus \mathcal{R}\subseteq E(\mathcal{S})$ and $$\mathcal{R}=\bigcup\limits_{i=1}^k \langle x_i\rangle$$ where ${\rm supp(T)}=\{x_1,x_2,\ldots,x_k\}$, $x_i*x_j=x_j$ and $\langle x_i\rangle^{\circ} \cap \langle x_j\rangle^{\circ}=\emptyset$
 for all $1\leq i<j\leq k$, and $\langle x\rangle^{\circ}$ denotes the subset of all non-idempotent elements in the finite cyclic semigroup $\langle x\rangle^{\circ}$, and
moreover, $\mathcal{I}(x_i)\equiv 1\pmod {\mathcal{P}(x_i)}$ and ${\rm v}_{x_i}(T)=\mathcal{I}(x_i)+\mathcal{P}(x_i)-2$ for every $i\in \{1,2,\ldots,k\}$.}

 \section{Concluding remarks}

 We remark that the value $|\mathcal{S}\setminus E(\mathcal{S})|+1$ is best possible to ensure that any $\mathcal{S}$-valued sequence of length  $|\mathcal{S}\setminus E(\mathcal{S})|+1$ is not (strongly) idempotent-product free, in the sense that $\mathcal{S}$ is a general semigroup. However, this value may be no longer best possible for a particular kind of semigroups. Hence, we introduce the following two combinatorial constants for any semigroup $\mathcal{S}$.

\begin{definition}\label{definition I(S) and SI(S)} \ Let $\mathcal{S}$ be a nonempty semigroup (not necessarily finite).
We define $\textsc{I}(\mathcal{S})$, which is called the {\bf Erd\H{o}s-Burgess constant} of the semigroup $\mathcal{S}$, to be the least $\ell\in\mathbb{N}\cup \{\infty\}$ such that every $\mathcal{S}$-valued sequence $T$ of length $\ell$ is not (weakly) idempotent-product free,
and we define $\textsc{SI}(\mathcal{S})$, which is called the {\bf strong Erd\H{o}s-Burgess constant} of the semigroup $\mathcal{S}$, to be the least $\ell\in\mathbb{N}\cup \{\infty\}$ such that every $\mathcal{S}$-valued sequence of length $\ell$ is not strongly idempotent-product free. Formally, we can also define $$\textsc{I}(\mathcal{S})={\rm sup}\ \{|T|+1: T \mbox{ takes every idempotent-product free } \mathcal{S}\mbox{-valued sequence}\}$$ and $$\textsc{IS}(\mathcal{S})={\rm sup}\ \{|T|+1: T \mbox{ takes every strongly idempotent-product free } \mathcal{S}\mbox{-valued sequence}\}.$$
\end{definition}

Then we have the following.

\begin{prop}\label{proposition I(S) and SI(S)} \ Let $\mathcal{S}$ be a nonempty semigroup.

\noindent (i).  \ If $\textsc{I}(\mathcal{S})$ or $\textsc{SI}(\mathcal{S})$ is finite then $\langle x\rangle$ is finite for every element $x\in \mathcal{S}$;

\noindent (ii). \ $\textsc{I}(\mathcal{S})\leq \textsc{SI}(\mathcal{S})$, and if $\mathcal{S}$ is commutative then $\textsc{I}(\mathcal{S})=\textsc{SI}(\mathcal{S})$; In  particular, for the case $|\mathcal{S}\setminus E(\mathcal{S})|$ is finite, $\textsc{I}(\mathcal{S})= |\mathcal{S}\setminus E(\mathcal{S})|+1$ holds if, and only if, the semigroup $\mathcal{S}$ is given as in Theorem \ref{theorem extremal-sequence}.
\end{prop}

\begin{proof} \  Conclusion (ii) follows from the definition and Theorem \ref{theorem extremal-sequence} readily.

(i). Suppose to the contrary the there exists some element $x\in \mathcal{S}$ such that $\langle x\rangle$ is infinite. Then the semigroup $\langle x\rangle$ is isomorphic the additive semigroup $\mathbb{N}^{+}$. The idempotent-product free sequence $x^{[\ell]}$ of arbitrarily great length $\ell\in \mathbb{N}$ gives the contradiction.
\end{proof}

The prerequisite that $\langle x\rangle$ is finite for every element $x\in \mathcal{S}$, is necessary for   $\textsc{I}(\mathcal{S})$ ($\textsc{SI}(\mathcal{S})$) being finite but not sufficient. For example, take a semigroup
\begin{equation}\label{equation G exampe}
\mathcal{S}=\langle \{x_i:i\in \mathbb{N}\}\rangle
\end{equation}
 where $x_i*x_j=x_j*x_i=x_j$ for any $1\leq i<j$, and
where $\langle x_t\rangle$ is a finite cyclic group of order $t+1$ for $t\in \mathbb{N}$. It is not hard to check that $x_1x_2\cdots x_k$ is an idempotent-product free  $\mathcal{S}$-valued sequence of length $k$ for any $k\in \mathbb{N}$, which gives that the infinity of $\textsc{I}(\mathcal{S})$ ($\textsc{SI}(\mathcal{S})$).

Hence, the following problems would be interesting.

\noindent \textbf{Problem 1.} \ {\sl Let $\mathcal{S}$ be a nonempty semigroup. Does there exist sufficient and necessary conditions to decide whether $\textsc{I}(\mathcal{S})$ ($\textsc{SI}(\mathcal{S})$) is finite or not?  }

 \noindent \textbf{Problem 2.} \ {\sl Let $\mathcal{S}$ be a nonempty semigroup. Does there exist sufficient and necessary conditions to decide whether $\textsc{I}(\mathcal{S})=\textsc{SI}(\mathcal{S})$ or not?}
 
 One thing worth remarking is that $\textsc{I}(\mathcal{S})$ is finite does not imply that $\textsc{SI}(\mathcal{S})$ is finite. For example, take the semigroup
$\mathcal{S}=\langle \{x_i:i\in \mathbb{N}\}\rangle\cup \{0_{\mathcal{S}}\}$ with zero element
 where $x_i*x_j=x_j$ and $x_j*x_i=0_{\mathcal{S}}$ for any $1\leq i<j$, a and
where $\langle x_t\rangle$ is a finite cyclic group of order some fixed integer $m>2$ for all $t\in \mathbb{N}$. It is easy to check that $\textsc{I}(\mathcal{S})=m$ and $\textsc{SI}(\mathcal{S})$ is infinite.

 \noindent \textbf{Problem 3.} \ {\sl Let $\mathcal{S}$ be a nonempty semigroup such that $|\mathcal{S}\setminus E(\mathcal{S})|$ is finite.  Find the sufficient and necessary conditions to decide whether $\textsc{SI}(\mathcal{S})=|\mathcal{S}\setminus E(\mathcal{S})|+1$.
Moreover, in case that $\textsc{SI}(\mathcal{S})=|\mathcal{S}\setminus E(\mathcal{S})|+1$,
determine the structure of the strongly idempotent-product free $\mathcal{S}$-valued sequences of length $|\mathcal{S}\setminus E(\mathcal{S})|$.

We remark that the above Problem 3 is in fact the inverse problem of Gillam, Hall and Williams (see Proposition \ref{proposition inifinity}).

 \noindent \textbf{Problem 4.} \ {\sl For some important kind of semigroup $\mathcal{S}$, determine the values of $\textsc{I}(\mathcal{S})$ and $\textsc{SI}(\mathcal{S})$.}

In the case that the semigroup $\mathcal{S}$ is commutative, the (strong variant is the same as shown in Proposition \ref{proposition I(S) and SI(S)}) Erd\H{o}s-Burgess constant seems to be closely related to a classical combinatorial constant, the {\bf Davenport constant} originated from K. Rogers \cite{Rogers}.  Davenport constant is the most important constant in Zero-sum Theory which has been extensively investigated for abelian groups since
 the 1960s (see\cite{EmdeBoas,Gao,GRuzsa, GS, Olson1}), and recently was also studied for commutative semigroups (see \cite{AGW,wangDavenportII,wangAddtiveirreducible,gaowang,gaowangII,wang-zhang-wang-qu,wang-zhang-qu},
 and P. 110 in \cite{GH}). For the readers' convenience, we state the definition of Davenport constant for commutative semigroups below.

\begin{definition} (\cite{wangDavenportII,wangAddtiveirreducible,gaowang}) \ Let $\mathcal{S}$ be a commutative semigroup.
Define $\mathsf D(\mathcal{S})$ to be  the least $\ell\in \mathbb{N}\cup \{\infty\}$ such that every $\mathcal{S}$-valued sequence
$T$ of
length at least $\ell$ contains a proper subsequence $T'$ ($T'\neq T$) the product whose terms is  equal to the product of all terms in $T$. \end{definition}

It is easy to see that for the case that $\mathcal{S}$ is an abelian group, both constants really mean the same thing, i.e., $\textsc{I}(\mathcal{S})=\mathsf D(\mathcal{S})$. While, for the case that the commutative semigroup  $\mathcal{S}$ is not a group, both $\textsc{I}(\mathcal{S})<\mathsf D(\mathcal{S})$ and  $\textsc{I}(\mathcal{S})>\mathsf D(\mathcal{S})$ could happen, which can be noticed from the following example.

\noindent \textbf{Example.} \ Take a commutative semigroup $\mathcal{S}=\langle x_1\rangle\cup \langle x_2\rangle$ where $\langle x_1\rangle$ is a finite cyclic group and $\langle x_2\rangle$ is a finite cyclic nilsemigroup with $x_1*x_2=x_2*x_1=x_2$ and $|\langle x_1\rangle|=n_1$ and $|\langle x_2\rangle|=n_2$. Then we check that $\textsc{I}(\mathcal{S})=(n_1-1)+(n_2-1)+1$ and $\mathsf D(\mathcal{S})=\max(n_1,n_2+1)$. By taking proper $n_1, n_2$, we have that both $\textsc{I}(\mathcal{S})<\mathsf D(\mathcal{S})$ and $\textsc{I}(\mathcal{S})>\mathsf D(\mathcal{S})$ could happen.

Therefore,  we close this manuscript by proposing the following problem.

\noindent \textbf{Problem 5.} \ {\sl Let $\mathcal{S}$ be a commutative semigroup. Does there exist any relationship between the Erd\H{o}s-Burgess constant $\textsc{I}(\mathcal{S})$ and the Davenport constant $\mathsf D(\mathcal{S})$?}

\bigskip

\noindent {\bf Acknowledgements}

\noindent  This work is supported by NSFC (11301381, 11271207), Science and Technology Development Fund of Tianjin Higher
Institutions (20121003).


\begin{thebibliography}{99}

\bibitem{AGW} S.D. Adhikari, W. Gao and G. Wang, \emph{Erd\H{o}s-Ginzburg-Ziv theorem for finite commutative semigroups}, Semigroup Forum, \textbf{88} (2014)  555--568.

\bibitem{Burgess69} D.A. Burgess, \emph{A problem on semi-groups}, Studia Sci. Math. Hungar., \textbf{4} (1969) 9--11.

\bibitem{EmdeBoas} P. van Emde Boas and D. Kruyswijk, \emph{ A combinatorial problem on finite abelian groups, 3,} Report ZW
1969-008, Stichting Math. Centrum, Amsterdam.



\bibitem{Gao} W. Gao, \emph{On Davenport's constant of finite abelian groups with rank three,} Discrete Math.,  \textbf{222} (2000) 111--124.




\bibitem{GRuzsa} A. Geroldinger, \emph{Additive Group Theory and Non-unique Factorizations,} 1--86 in: A. Geroldinger and I. Ruzsa (Eds.), Combinatorial Number
Theory and Additive Group Theory (Advanced Courses in Mathematics-CRM Barcelona), Birkh\"{a}user, Basel, 2009.

\bibitem{GH} A. Geroldinger and F. Halter-Koch, \emph{Non-Unique
Factorizations. Algebraic, Combinatorial and Analytic Theory,}
Pure and Applied Mathematics, vol. 278, Chapman $\&$ Hall/CRC,
2006.

\bibitem{GS} A. Geroldinger and R. Schneider, \emph{On  Davenport' s  constant,}  J. Combin. Theory Ser. A, \textbf{61} (1992)  147--152






\bibitem{Gillam72} D.W.H. Gillam, T.E. Hall and N.H. Williams, \emph{On finite semigroups and idempotents}, Bull. Lond. Math. Soc., \textbf{4} (1972) 143--144.


\bibitem{Grillet semigroup} P.A. Grillet, \emph{Semigroups: An introduction to the Structure Theory,} Dekker, New York, 1995.

\bibitem{Grillet monograph} P.A. Grillet, \emph{Commutative Semigroups,} Kluwer Academic Publishers, 2001.

\bibitem{Grynkiewicz13}    D.J. Grynkiewicz, \emph{The large Davenport constant II: General upper bounds}, J. Pure Appl. Algebra, \textbf{217} (2013) 2221--2246.

\bibitem{Olson1} J.E. Olson,  \emph{A Combinatorial Problem on Finite Abelian Groups, I}, J. Number Theory, \textbf{1} (1969) 8--10.

\bibitem{Rogers}
K. Rogers, \emph{A Combinatorial problem in Abelian groups,} Math. Proc. Cambridge Philos. Soc., \textbf{59} (1963) 559--562.

\bibitem{wangDavenportII}  G. Wang, \emph{Davenport constant for semigroups II,}  J. Number Theory, \textbf{153} (2015) 124--134.

\bibitem{wangAddtiveirreducible}  G. Wang, \emph{Additively irreducible sequences in commutative semigroups,}  arXiv:1504.06818.

\bibitem{gaowang} G. Wang and W. Gao,
\emph{Davenport constant for semigroups,} Semigroup Forum,
\textbf{76} (2008) 234--238.

\bibitem{gaowangII} G. Wang and W. Gao,
Davenport constant of the multiplicative semigroup of the ring $\mathbb{Z}_{n_1}\oplus\cdots\oplus \mathbb{Z}_{n_r}$, arXiv:1603.06030.

\bibitem{wang-zhang-wang-qu}  H.L. Wang, L.Z. Zhang, Q.H. Wang and Y.K. Qu,  \emph{Davenport constant of the multiplicative semigroup of the quotient ring $\frac{\F_p[x]}{\langle f(x)\rangle}$,} International Journal of Number Theory, in press,
 DOI: 10.1142/S1793042116500433.


 \bibitem{wang-zhang-qu} L.Z. Zhang, H.L. Wang and Y.K. Qu,  \emph{A problem of Wang on Davenport constant for the multiplicative semigroup of the quotient ring of $\F_2[x]$},  arXiv:1507.03182.


\end{thebibliography}
\end{document}